\documentclass[1 [leqno,11pt]{amsart}
\usepackage{amssymb, amsmath}
\let\pa=\partial
\let\e=\varepsilon
\def\N{\mathop{\mathbb N\kern 0pt}\nolimits}
\def\Q{\mathop{\mathbb Q\kern 0pt}\nolimits}
\def\R{\mathop{\mathbb R\kern 0pt}\nolimits}
\def\om{\omega}
\def\th{\theta}
\def\al{\alpha}
\def\e{\epsilon}

\def\ot{\omega^\theta}
\def\oto{\omega^{\theta,(1)}}
\def\ott{\omega^{\theta,(2)}}
\def\otl{\omega^{\theta,(\ell)}}
\def\ur{u^r}

\def\uz{u^z}
\def\cX{\mathcal{X}}
\def\cL{\mathcal{L}}
\def\cH{\mathcal{H}}
\def\cC{\mathcal{C}}

\def\dive{\mathop{\rm div}\nolimits}
\def\curl{\mathop{\rm curl}\nolimits}
\def\eqdefa{\buildrel\hbox{\footnotesize def}\over =}

\newtheorem{theorem}{Theorem}[section]
\newtheorem{lemma}{Lemma}[section]
\newtheorem{proposition}{Proposition}[section]

\newtheorem{definition}{Definition}[section]

\newtheorem{lem}{Lemma}[section]
\newtheorem{rmk}{Remark}[section]

\newtheorem{prop}{Proposition}[section]
\numberwithin{equation}{section}

\begin{document}
\title[Uniqueness of axisymmetric viscous flows]
{Uniqueness of axisymmetric viscous flows originating from positive linear combinations of circular vortex filaments}

\author[G. L\'evy]{Guillaume L\'evy}
\address[G. L\'evy]{Laboratoire Jacques-Louis Lions, UMR 7598, Universit\'{e} Pierre et Marie Curie, 75252 Paris Cedex 05, France.} \email{levy@ljll.math.upmc.fr}
\author[Y. Liu]{Yanlin Liu}
\address[Y. Liu]{Department of Mathematical Sciences, University of Science and Technology of China, Hefei 230026, CHINA,
and Academy of Mathematics $\&$ Systems Science, Chinese Academy of
Sciences, Beijing 100190, CHINA.} \email{liuyanlin3.14@126.com}

\begin{abstract}
Following the recent papers \cite{GS15} and \cite{GS16} by T. Gallay and V. \u{S}ver\'ak, in the line of work initiated by H. Feng and V. \u{S}ver\'ak
in their paper \cite{FengSverak}, we prove the uniqueness of a solution of the axisymmetric Navier-Stokes equations without swirl when the initial data is a 
positive linear combination of Dirac masses.
\end{abstract}

\maketitle

\noindent {\sl Keywords: Axisymmetric Navier-Stokes, fluid mechanics, vortex filaments}

\vskip 0.2cm
\noindent {\sl AMS Subject Classification (2000):} 35Q30, 76D03  \
\setcounter{equation}{0}
\section{Introduction}

In 3-D ideal fluids, a vortex ring is an axisymmetric flow whose vorticity is entirely concentrated in
a solid torus, which moves with constant speed along the symmetry axis. See \cite{AS89, F70, FB74, FT81}
for the existence of vortex ring solutions to the 3-D Euler equations.

However, for viscous fluids, the vortex ring solutions can not exist, since all localized structures
will be spread out by diffusion. Thus it is natural to consider the Navier-Stokes equations with
a vortex filament, and more generally with positive linear combinations of circular vortex filaments
which have a common axis of symmetry as initial data.

To state this precisely,
let us start with the Navier-Stokes equations in $\R^3$
\begin{equation}\label{NSu}
 \pa_t u+u\cdot\nabla u-\Delta u+\nabla p=0,
\quad \dive u=0,\qquad(t,x)\in\R^+\times\R^3,
\end{equation}
where $u(t,x)=(u^1,u^2,u^3)$ stands for the velocity field and $p$
the scalar pressure function of the fluid, which guarantees that the velocity field remains divergence free.

In the following, we restrict ourselves to the axisymmetric solutions without swirl 
of \eqref{NSu}, for which the velocity field $u$ and the vorticity $\omega\eqdefa \curl u$ take the particular form
$$u(t,x)=u^r(t,r,z)e_r+u^z(t,r,z)e_z,\quad
\omega(t,x)=\ot(t,r,z)e_\theta,$$
where $(r,\theta,z)$ denotes the cylindrical coordinates in $\R^3$
so that $x=(r\cos\theta,r\sin\theta,z)$,  and
$$e_r=(\cos\theta,\sin\theta,0),\ e_\theta=(-\sin\theta,\cos\theta,0),\ e_z=(0,0,1),\ r=\sqrt{x_1^2+x_2^2}.$$

In view of \cite{GS15}, we equip the half-plane $\Omega=\{(r,z)|r>0,z\in\R\}$
with the measure $drdz$. More precisely, for any measurable function $f:\Omega\rightarrow\R$, we denote
$$
\|f\|_{L^p(\Omega)}\eqdefa\Bigl(\int_\Omega|f(r,z)|^p drdz\Bigr)^{\frac 1p}<\infty,\  1\leq p<\infty,
$$
and $\|f\|_{L^\infty(\Omega)}$ to be the essential supremum of $|f|$ on $\Omega$.
For notational simplicity, we shall always denote a generic point in $\Omega$ by $x=(r,z)$.

Recalling the axisymmetric Biot-Savart law discussed in Section $2$ of \cite{GS15},
we know that for any given $\ot\in L^1(\Omega)\cap L^\infty(\Omega)$
which vanishes on $r=0$, the linear elliptic system
\begin{equation*}
\left\{
\begin{split}
& \pa_r u^r+\frac 1r u^r+\pa_z u^z=0,\quad \pa_z\ur-\pa_r\uz=\ot, \quad\mbox{on } \Omega,\\
& \ur|_{r=0}=0,\quad \pa_r\uz|_{r=0} =0,
\end{split}
\right.
\end{equation*}
has a unique solution $(\ur,\uz)\in C(\Omega)^2$ vanishing at infinity.
We denote this solution by $u=BS[\ot]$.
Hence we only need to study the equation for $\ot$:
\begin{equation}\label{NSom}
\pa_t \ot+(u^r\pa_r+u^z\pa_z) \ot-\frac{u^r \omega^\theta}{r}
=(\pa_r^2+\pa_z^2+\frac 1r\pa_r-\frac{1}{r^2})\omega^\theta.
\end{equation}

Now let us discuss the initial condition.
We first recall from \cite{GS15} that, the axisymmetric vorticity equation \eqref{NSom} is globally well-posed
whenever the initial vorticity is in $L^1(\Omega)$.
As a natural extension, then they considered the initial vorticity in $\mathcal{M}(\Omega)$,
which denotes the set of all real-valued finite regular measures on $\Omega$,
equipped with the total variation norm
$$\|\mu\|_{\rm{tv}}\eqdefa\sup\Bigl\{\int_\Omega\phi\, d\mu\,\Big|\,\phi\in C_0(\Omega),
\|\phi\|_{L^\infty(\Omega)}\leq 1\Bigr\},$$
where $C_0(\Omega)$ denotes the set of all real-valued continuous functions on $\Omega$
that vanishes at infinity and on the boundary $\pa\Omega$. It is also proved in \cite{GS15} that
\eqref{NSom} is globally well-posed if the initial vorticity $\mu$ is in $\mathcal{M}(\Omega)$
whose atomic part is small enough.

As mentioned in the first paragraph of the introduction, we focus here on the particular case 
$$\mu=\sum\limits_{i=1}^n\alpha_i\delta_{x_i},$$
where $\alpha_i$ is some positive constant and $\delta_{x_i}$ is the Dirac mass at point $x_i=(r_i,z_i)\in\Omega$
with $r_i>0$. Such a $\mu$ is purely atomic and we deduce from \cite{GS15} that
\eqref{NSom} is global well-posed provided that 
$$\|\mu\|_{\rm{tv}}=\sum\limits_{i=1}^n\alpha_i$$
is small enough. On the other hand, for arbitrary positive values of $\alpha_i$,
\cite{FengSverak} gives the existence of a global mild solution, and \cite{GS16} proves
the uniqueness when $n=1$. In this paper, we prove the uniqueness for general $n$.
Our result can be stated as follows:
\begin{theorem}\label{thmmain}
{\sl Fix an integer $n$.
 Let 
 $$\mu=\sum\limits_{i=1}^n\alpha_i\delta_{x_i},$$
where $\alpha_i$ is some positive constant and $\delta_{x_i}$ is the Dirac mass at point $x_i=(r_i,z_i)\in\Omega$
with $r_i>0$.
 Then \eqref{NSom} has a unique global solution $\ot$ in 
 $\cC\bigl(]0,\infty[, L^1(\Omega) \cap L^{\infty}(\Omega)\bigr)$
 in the mild sense (see Definition \ref{defmild}),
 satisfying
 \begin{equation}\label{thmcondi}
 \sup_{t>0}\|\ot(t)\|_{L^1(\Omega)} < \infty,\quad\mbox{and}\quad \ot(t) dr dz \rightharpoonup \mu  \quad\mbox{as}\quad t \rightarrow 0.
 \end{equation}
 Moreover, there exists some constant $C_0$ depending only on
 $(\alpha_i,x_i)_{i=1}^n$,
such that whenever $\sqrt{t}\leq\frac12\min\limits_{1\leq i<j\leq n}\bigl\{|x_i-x_j|,r_i\bigr\}$, there holds the following short time estimate:
 \begin{equation}\label{ShortTimeAsymptotics}
   \Bigl\|\ot(t,\cdot) - \frac{1}{4\pi t}\sum_{i=1}^n \alpha_i e^{- \frac{|\cdot-x_i|^2}{4t}}\Bigr\|_{L^1(\Omega)}
   \leq C_0\sqrt t |\ln t|.
 \end{equation}
}\end{theorem}

Let us end this section with some notations. We use $C$ (resp. $C_0$) to denote some absolute positive constant
(resp. some positive constant depending on $(\alpha_i,x_i)_{i=1}^n$), which may be different in each occurrence.
$f\lesssim g$ means that there exists some constant $C$ such that $f\leq Cg$.
For a Banach space $B$, we shall use the shorthand $\|u\|_{L^p_TB}$ for the norm $\bigl\|\|u(t,\cdot)\|_B\bigr\|_{L^p(0,T)}$.

\section{Decomposition of the solution}

In order to use the uniqueness result for the case when the initial measure
 is one single Dirac mass which has been proved in \cite{GS16},
a natural thought is to decompose the solution into $n$ parts: 
$$\ot=\sum\limits_{i=1}^n\ot_i,$$
 according to the decomposition of the
initial measure 
$$\mu=\sum\limits_{i=1}^n\alpha_i\delta_{x_i}.$$ 
The nonlinearity of the equation \eqref{NSom} renders this idea nontrivial to implement.
The strategy is to use the fundamental solution of some advection-diffusion equation.
This will be done in the first subsection.

The purpose of the second subsection will be to show that, at least for short times, $\ot_i$ is very close -- in the $L^1(\Omega)$ sense -- to the Oseen vortex located at $x_i$ 
with circulation $\al_i$. 
This goal will be achieved using self-similar variables around the point $x_i$.

\subsection{The linear semigroup and the trace of the solution at initial time}

Let us denote by $\bigl(S(t)\bigr)_{t\geq0}$ the evolution semigroup
defined by the linearized system of \eqref{NSom}, namely
\begin{equation}\label{linear}
\left\{
\begin{split}
& \pa_t\ot-\bigl(\pa_r^2+\pa_z^2+\frac 1r \pa_r-\frac{1}{r^2}\bigr)\ot=0,   \quad(t,r,z)\in\R^+\times\Omega,\\
& \ot|_{r=0}=0,\quad \ot|_{t=0} =\ot_0.
\end{split}
\right.
\end{equation}
One can see Section $3$ of \cite{GS15} for a detailed study of this semigroup.

By using $\bigl(S(t)\bigr)_{t\geq0}$, we can define the mild solutions of \eqref{NSom} in the following way:
\begin{definition}\label{defmild}
{\sl Let $T>0$, we say that
 $\ot\in C\bigl(]0,T[, L^1(\Omega) \cap L^{\infty}(\Omega)\bigr)$
 is a mild solution of \eqref{NSom} on $]0,T[$, if for any $0 < t_0 < t < T$, there holds the following integral equation
 \begin{equation}\label{integraleqt}
  \ot(t) = S(t-t_0)\ot(t_0) - \int_{t_0}^t S(t-s)\dive_*
  \bigl(u(s)\ot(s)\bigr) \, ds.
 \end{equation}
 Here $u= \rm{BS}[\ot]$ and $\dive_*
  \bigl(u\ot\bigr)\eqdefa\pa_r(\ur\ot)+\pa_z(\uz\ot)$.
}\end{definition}

Before proceeding further, let us recall some {\it a priori} estimates for the mild solution.
\begin{lem}
{\sl Let $\ot$ be a mild solution of \eqref{NSom} on $(0,T)$ satisfying \eqref{thmcondi},
$u= \rm{BS}[\ot]$. It is shown in Estimates $(2.13),~(2.14)$ of \cite{GS16} that,
for any $t\in]0,T[,$ and any $k,\ell\in\N$, there holds
\begin{equation}\label{apriorimild}
t^{k+\frac{\ell}2+\frac12}\|\pa_t^k\nabla_x^\ell u(t)\|_{L^\infty(\Omega)}+t^{\frac32}\|\nabla\ot(t)\|_{L^\infty(\Omega)}\leq C_0.
\end{equation}
Moreover, we can deduce from Estimate $(9)$ of \cite{GS15} that
\begin{equation}\label{limitott0}
\lim\limits_{t\rightarrow 0}t^{1-\frac1p}\|\ot(t)\|_{L^p(\Omega)}=0,\quad\mbox{for any}\quad
1<p\leq\infty.
\end{equation}
}\end{lem}

Combining the conclusions of Corollary $2.9,~2.10$ and Remark $2.11$ in \cite{GS16}, we prove the following.
\begin{prop}\label{propconvergence}
{\sl For any $T>0$, if $\ot\in C\bigl((0,T), L^1(\Omega) \cap L^{\infty}(\Omega)\bigr)$
is a mild solution of \eqref{NSom} on $(0,T)$ satisfying \eqref{thmcondi}, then for any $t\in(0,T)$
and $(r,z)\in\Omega$, we have
\begin{equation}\label{boundmild}
\ot(t,r,z)\geq0,\quad \|\ot(t)\|_{L^1(\Omega)}
\leq\|\mu\|_{\rm{tv}}\quad\mbox{and}\quad
\lim\limits_{t\rightarrow 0}\|\ot(t)\|_{L^1(\Omega)}=\|\mu\|_{\rm{tv}}.
\end{equation}
Moreover, for any bounded and continuous function $\phi$ on $\Omega$, there holds the convergence
\begin{equation}\label{convergence}
\int_{\Omega}\phi(r,z)\ot(t,r,z)\, drdz\rightarrow
\int_{\Omega}\phi\, d\mu,\quad\mbox{as}\quad t\rightarrow0.
\end{equation}
}\end{prop}

Noting that although the initial measure $\mu$ is no longer a single Dirac mass
as considered in \cite{GS16}, it is still supported in
$\bigl[\min\limits_{1\leq i\leq n} r_i,\max\limits_{1\leq i\leq n} r_i\bigr]\times\R$.
Thus the estimates of Proposition $3.1,~3.3$ and then Lemma $3.8$ in \cite{GS16} still hold for the case here.
Precisely, we have
\begin{equation}\label{boundur}
\int_0^T\|\ur(t)/r\|_{L^\infty(\Omega)}\, dt\leq C_0.
\end{equation}

Next, let us state a particular case of Aronson's pioneering work \cite{Aronson}
on the fundamental solution of parabolic equations,
which will be a key ingredient in our decomposition.
\begin{prop}[Proposition $3.9$ of \cite{GS16}]\label{prop3.9ofGS16}
{\sl Assume that $U,~V:(0,T)\times\R^3\rightarrow\R^3$ are continuous functions such that
$\dive U(t,\cdot)=0,$ for all $t\in(0,T)$ and
$$\sup_{0<t<T}t^{\frac12}\|U(t,\cdot)\|_{L^\infty(\R^3)}=K_1<\infty,
\quad \int_0^T\|V(t,\cdot)\|_{L^\infty(\R^3)}\, dt=K_2<\infty.$$
Then the regular solutions of the following type advection-diffusion equation
\begin{equation}\label{a-deqt}
\pa_t f+U\cdot\nabla f-V f=\Delta f,\quad x\in\R^3,\quad t\in(0,T),
\end{equation}
can be represented in the following way:
$$f(t,x)=\int_{\R^3}\Phi_{U,V}(t,x;s,y)f(s,y)\, dy,\quad x\in\R^3,\ 0<s<t<T,$$
where $\Phi_{U,V}$ is the (uniquely defined) fundamental solution,
which is H\"older continuous in space and time,
and satisfies, for all $x,y\in\R^3$ and $0<s<t<T$, that
\begin{equation}\label{boundPhi}
0<\Phi_{U,V}(t,x;s,y)\leq \frac{C}{(t-s)^{\frac 32}}\exp
\Bigl(-\frac{|x-y|^2}{4(t-s)}+K_1\frac{|x-y|}{\sqrt{t-s}}+K_2\Bigr).
\end{equation}
}\end{prop}

It is easy to derive the evolution equation for $\omega=\ot(t,r,z)e_\th$ from \eqref{NSu} that
\begin{equation}
\pa_t \om+u\cdot\nabla \om-r^{-1}\ur\om=\Delta \om,\quad x\in\R^3,\quad t\in(0,T),
\end{equation}
which is exactly of the form \eqref{a-deqt} with $U=u,~V=r^{-1}\ur$. In view of \eqref{apriorimild} and \eqref{boundur},
the conditions of Proposition \ref{prop3.9ofGS16} are satisfied. Thus this $\om$ can be represented as
$$\om(t,x)=\int_{\R^3}\Phi(t,x;s,y)\om(s,y)\, dy,\quad x\in\R^3,\ 0<s<t<T.$$
From which, we can deduce that $\ot$ satisfies
\begin{equation}\label{otdecomposition1}
\ot(t,r,z)=\int_{\Omega}\widetilde{\Phi}(t,r,z;s,r',z')\ot(s,r',z')\, dr'dz',\quad
 0<s<t<T,
\end{equation}
where
$$\widetilde{\Phi}(t,r,z;s,r',z')=\int_{-\pi}^{\pi}
\Phi\bigl(t,(r,0,z);s,(r'\cos\th,r'\sin\th,z')\bigr)\cdot r'\cos\th\,d\th.$$
Using the Gaussian upper bound \eqref{boundPhi} of the fundamental solution $\Phi$, we get
\begin{lem}[Lemma $3.10$ of \cite{GS16}]
{\sl For any $\eta\in]0,1[$ and $0<s<t<T$, there exists some positive constant $C_{\eta,\al}$
depending only on the choice of $\eta$ and $(\al_i)_{i=1}^n$, such that
\begin{equation}\label{boundtildePhi}
0<\widetilde{\Phi}(t,r,z;s,r',z')\leq\frac{C_{\eta,\al}}{t-s}\Bigl|\frac{r'}{r}\Bigr|^{\frac12}
\widetilde{H}\Bigl(\frac{t-s}{(1-\eta)rr'}\Bigr)e^{-\frac{1-\eta}{4(t-s)}
\bigl((r-r')^2+(z-z')^2\bigr)},
\end{equation}
where $\widetilde{H}:(0,\infty)\rightarrow\R$ is decreasing with
$\widetilde{H}(\tau)\rightarrow 1$ as $\tau\rightarrow 0$ and
$\widetilde{H}(\tau)\thicksim 1/\sqrt{\pi\tau}$ as $\tau\rightarrow \infty$.
}\end{lem}

Let us write \eqref{otdecomposition1} in the following way
\begin{align*}
\ot(t,r,z)=\int_{\Omega}&\widetilde{\Phi}(t,r,z;0,r',z')\ot(s,r',z')\, dr'dz'\\
&+\int_{\Omega}\bigl(\widetilde{\Phi}(t,r,z;s,r',z')-\widetilde{\Phi}(t,r,z;0,r',z')\bigr)\ot(s,r',z')\, dr'dz'.
\end{align*}
In view of the H\"older continuity and Gaussian upper bound \eqref{boundPhi}
of the fundamental solution $\Phi$, we deduce that $\widetilde{\Phi}$ is continuous whenever $0<s<t<T$.
Combining this with the facts that $\widetilde{\Phi}$ is bounded as shown
in \eqref{boundtildePhi}, and  $\|\ot(t)\|_{L^1(\Omega)}
\leq\|\mu\|_{\rm{tv}}$ as shown in \eqref{boundmild}, we know the second integral in the right-hand side converges to $0$
as $s$ tends to $0$. On the other hand, since $\widetilde{\Phi}$ is continuous and bounded, we can use \eqref{convergence} to derive
the limit of the first integral as $s$ tends to $0$, and we finally obtain the following useful representation:
\begin{equation*}
\ot(t,r,z)=\int_{\Omega}\widetilde{\Phi}(t,r,z;0,r',z')\,d\mu.
\end{equation*}
Recalling $\mu=\sum\limits_{i=1}^n\alpha_i\delta_{x_i}$, we can obtain the decomposition for $\ot$ as follows:
\begin{equation}\label{decomototi}
\ot(t,r,z)=\sum\limits_{i=1}^n\ot_i(t,r,z),\quad \mbox{where}\quad
\ot_i(t,r,z)=\al_i\widetilde{\Phi}(t,r,z;0,r_i,z_i),
\end{equation}
and the corresponding decomposition for $u=BS[\ot]$:
\begin{equation}\label{decomuui}
u(t,r,z)=\sum\limits_{i=1}^n u_i(t,r,z),\quad \mbox{where}\quad
u_i=BS[\ot_i].
\end{equation}

It is easy to see that $\ot_i
\in C\bigl(]0,T[, L^1(\Omega) \cap L^{\infty}(\Omega)\bigr)$ is a mild solution of
\begin{equation}\label{eqtoti}
\left\{
\begin{split}
& \pa_t\ot_i+u\cdot\nabla\ot_i-\bigl(\pa_r^2+\pa_z^2+\frac 1r \pa_r-\frac{1}{r^2}\bigr)\ot_i=0,   \quad(t,r,z)\in]0,T[\times\Omega,\\
& \ot_i\rightharpoonup \alpha_i\delta_{x_i}\quad\mbox{as}\quad t\rightarrow 0.
\end{split}
\right.
\end{equation}
Moreover, we have the following estimates for $\ot_i$.
\begin{prop}
{\sl $i)$ For any $\eta\in]0,1[,~(r,z)\in\Omega$ and $0<t<T$, we have
\begin{equation}\label{boundoti}
0<\ot_i(t,r,z)\leq\frac{C_{\eta,\al}}{t}
e^{-\frac{1-\eta}{4t}\bigl((r-r_i)^2+(z-z_i)^2\bigr)}.
\end{equation}
\begin{equation}\label{otiL1limit}
\|\ot_i(t)\|_{L^1(\Omega)}
\leq\|\mu\|_{\rm{tv}}\quad\mbox{and}\quad
\lim\limits_{t\rightarrow 0}\|\ot_i(t)\|_{L^1(\Omega)}=\al_i.
\end{equation}
$ii)$ There exists some positive time $t_1<T$, such that for any $0<t<t_1$, there holds
\begin{equation}\label{boundnablaoti}
t^{\frac32}\|\nabla\ot_i(t)\|_{L^\infty(\Omega)}\leq C_0.
\end{equation}
}\end{prop}
\begin{proof}
$i)$ Using \eqref{boundtildePhi}, we immediately get
\begin{equation}\label{boundoti1}
0<\ot_i(t,r,z)\leq\frac{C_{\eta,\al}}{t}\Bigl|\frac{r_i}{r}\Bigr|^{\frac12}
\widetilde{H}\Bigl(\frac{t}{(1-\eta)rr_i}\Bigr)e^{-\frac{1-\eta}{4t}
\bigl((r-r_i)^2+(z-z_i)^2\bigr)}.
\end{equation}

When $2r\leq r_i$, using the facts
$\widetilde{H}(\tau)\leq 1/\sqrt{\pi\tau}$ and $2|r_i-r|\geq r_i$ in this case gives
$$\Bigl|\frac{r_i}{r}\Bigr|^{\frac12}
\widetilde{H}\Bigl(\frac{t}{(1-\eta)rr_i}\Bigr)
\leq \frac{r_i}{\sqrt\pi}\bigl(\frac{1-\eta}{t}\bigr)^{\frac12}
\leq C_{\eta,\al}\cdot e^{\frac{\eta(1-\eta)}{4t}(r-r_i)^2}.$$
Substituting this into \eqref{boundoti1}, and noting the fact that,
when $\eta$ runs over $]0,1[$, $(1-\eta)^2$ also runs over $]0,1[$, gives exactly \eqref{boundoti} in this case.

And when $2r> r_i$, \eqref{boundoti} follows by simply bounding $\widetilde{H}$ by $1$ in \eqref{boundoti1}.

To prove \eqref{otiL1limit}, notice that $\ot_i>0$ and $\ot=\sum\limits_{i=1}^n\ot_i$, we have
\begin{equation}\label{boundot2}
\sum\limits_{i=1}^n\|\ot_i(t)\|_{L^1(\Omega)}=\|\ot(t)\|_{L^1(\Omega)}
\leq\|\mu\|_{\rm{tv}},\quad\forall t\in]0,T[,
\end{equation}
which in particular implies $\|\ot_i(t)\|_{L^1(\Omega)}\leq\|\mu\|_{\rm{tv}}$.
By taking limit $t\rightarrow0$ in \eqref{boundot2}, we obtain
$$\sum\limits_{i=1}^n\lim_{t\rightarrow0}\|\ot_i(t)\|_{L^1(\Omega)}
=\lim_{t\rightarrow0}\|\ot(t)\|_{L^1(\Omega)}
=\|\mu\|_{\rm{tv}}=\sum\limits_{i=1}^n\al_i.$$
On the other hand, the initial condition
$\ot_i\rightharpoonup \alpha_i\delta_{x_i}$ as $t\rightarrow 0$
implies
$$\lim_{t\rightarrow0}\|\ot_i(t)\|_{L^1(\Omega)}\geq\al_i.$$
Combining the above two sides, clearly there must hold
$$\lim\limits_{t\rightarrow 0}\|\ot_i(t)\|_{L^1(\Omega)}=\al_i.$$
$ii)$ For any $0<t<T$, we first write \eqref{eqtoti} in the integral form as
\begin{equation}\label{boundnablaoti1}
\ot_i(t) = S(t/2)\ot_i(t/2) - \int_{t/2}^t S(t-s)\dive_*
  \bigl(u(s)\ot_i(s)\bigr) \, ds.
\end{equation}
Then we need the following lemma, which is a particular case of
\begin{lem}
{\sl For any $1\leq p\leq q\leq \infty,$ and $f(r,z)\in L^p(\Omega)$,
there holds
\begin{equation}\label{boundnablaoti2}
\|{\nabla} S(t)f \|_{L^q(\Omega)}
\leq\frac{C}{t^{\frac 12+\frac 1p-\frac 1q}}\| f\|_{L^p(\Omega)},
\end{equation}
}\end{lem}
Using \eqref{boundnablaoti1} and \eqref{boundnablaoti2}, together with the bounds
\eqref{apriorimild} and \eqref{boundmild}, as well as the fact that $\ot_i\leq\ot$ point-wisely, we achieve
\begin{align*}
\|\nabla\ot_i(t)\|_{L^\infty(\Omega)}&\leq\frac{C}{t^{3/2}}\|\ot_i(t/2)\|_{L^1(\Omega)}
+\int_{\frac t2}^t\frac{C}{(t-s)^{1/2}}\bigl(\|\nabla u(s)\|_{L^\infty(\Omega)}\|\ot_i(s)\|_{L^\infty(\Omega)}\\
&\qquad\qquad\qquad\qquad\qquad\qquad\qquad\qquad\quad
+\|u(s)\|_{L^\infty(\Omega)}\|\nabla\ot_i(s)\|_{L^\infty(\Omega)}\bigr)\, ds\\
&\leq\frac{C_0}{t^{3/2}}+\int_{\frac t2}^t
\frac{C_0}{(t-s)^{1/2}}\Bigl(\frac{1}{s^2}+\frac{1}{\sqrt s}\cdot s^{\frac12}
\|u(s)\|_{L^\infty(\Omega)}\|\nabla\ot_i(s)\|_{L^\infty(\Omega)}\Bigr)\, ds\\
&\leq\frac{C_0}{t^{3/2}}+C_0\sup_{t/2<s<t}s^{\frac12}
\|u(s)\|_{L^\infty(\Omega)}\cdot\sup_{t/2<s<t}\|\nabla\ot_i(s)\|_{L^\infty(\Omega)}.
\end{align*}
Multiplying both sides by $t^{3/2}$, we get
$$t^{\frac32}\|\nabla\ot_i(t)\|_{L^\infty(\Omega)}\leq C_0+C_0
\sup_{t/2<s<t}s^{\frac12}\|u(s)\|_{L^\infty(\Omega)}\cdot
\sup_{t/2<s<t}s^{\frac32}\|\nabla\ot_i(s)\|_{L^\infty(\Omega)}.$$
Then taking supremum over $t$ leads to
\begin{equation}\label{boundnablaoti3}
\sup_{0<s<t}s^{\frac32}\|\nabla\ot_i(s)\|_{L^\infty(\Omega)}\leq C_0\bigl(1+
\sup_{0<s<t}s^{\frac12}\|u(s)\|_{L^\infty(\Omega)}\cdot
\sup_{0<s<t}s^{\frac32}\|\nabla\ot_i(s)\|_{L^\infty(\Omega)}\bigr).
\end{equation}
Noting that $u=BS[\ot]$, we can use Proposition $2.3$ of \cite{GS15} to obtain
$$\|u\|_{L^\infty(\Omega)}\leq C \|\ot\|_{L^1(\Omega)}^{\frac12}\|\ot\|_{L^\infty(\Omega)}^{\frac12},$$
which together with \eqref{limitott0} indicates that
\begin{equation}
\lim_{t\rightarrow 0}t^{\frac12}\|u(t)\|_{L^\infty(\Omega)}=0.
\end{equation}
Thus there exists some $t_1>0$, such that for any $s\in]0,t_1[$ and the $C_0$ in \eqref{boundnablaoti3}, there holds
$$C_0\cdot s^{\frac12}\|u(s)\|_{L^\infty(\Omega)}<\frac12,$$
which guarantees that the term $C_0\sup\limits_{0<s<t}s^{\frac12}\|u(s)\|_{L^\infty(\Omega)}\cdot
\sup\limits_{0<s<t}s^{\frac32}\|\nabla\ot_i(s)\|_{L^\infty(\Omega)}$ in \eqref{boundnablaoti3} can be absorbed by the left hand side.
This gives exactly the desired estimate \eqref{boundnablaoti}.
\end{proof}

\subsection{Self-similar variables}
In view of \eqref{boundoti}, we know that
$\ot_j$ concentrates in a self-similar way around $x_j$ for short time.
Thus it is very natural to introduce the self-similar variables:
\begin{equation}\label{defselfsimilarvariables}
R_j=\frac{r-r_j}{\sqrt t},\quad Z_j=\frac{z-z_j}{\sqrt t},\quad X_j=\frac{x-x_j}{\sqrt t}
\quad\mbox{and}\quad \e_j=\frac{\sqrt t}{r_j},\quad j=1,\cdots,n.
\end{equation}

Correspondingly, for any $j\in\{1,\cdots,n\},~t\in(0,T)$ and any $(r,z)\in\Omega $, we set
\begin{equation}\label{blowupotiui}
\ot_j(t,r,z)=\frac{\al_j}{t}f_j\Bigl(t,\frac{r-r_j}{\sqrt t},\frac{z-z_j}{\sqrt t}\Bigr),
\quad u_j(t,r,z)=\frac{\al_j}{\sqrt t}U_j\Bigl(t,\frac{r-r_j}{\sqrt t},\frac{z-z_j}{\sqrt t}\Bigr).
\end{equation}
In the new coordinates $(R_j,Z_j)$, the domain constraint $r>0$ translates into $r_j+\sqrt{t}R_j>0$,
which means that the rescaled vorticity $f_j(t,R_j,Z_j)$ is defined in the time-dependent domain
$$\Omega_{\e_j}\eqdefa\{(R_j,Z_j)\in\R^2\,|\,1+\e_j R_j>0\}.$$
Noting that $u_j=BS[\ot_j]$, thus $U_j$ can also be determined by $f_j$.
Recalling the subsection $4.2$ of \cite{GS16}, we have the following explicit representation
\begin{equation}\begin{split}\label{paraBS}
&U_j^r(X_j)=\frac{1}{2\pi}\int_{\Omega_{\e_j}}\sqrt{(1+\e_j R')(1+\e_j R_j)^{-1}}
F_1(\xi_j^2)\frac{Z_j-Z'}{|X_j-X'|^2} f_j(X')\,dX',\\
&U_j^z(X_j)=-\frac{1}{2\pi}\int_{\Omega_{\e_j}}\sqrt{(1+\e_j R')(1+\e_j R_j)^{-1}}
F_1(\xi_j^2)\frac{R_j-R'}{|X_j-X'|^2} f_j(X')\,dX'\\
&\qquad\qquad+\frac{\e_j}{4\pi}\int_{\Omega_{\e_j}}\sqrt{(1+\e_j R')(1+\e_j R_j)^{-3}}
\bigl(F_1(\xi_j^2)+F_2(\xi_j^2)\bigr) f_j(X')\,dX',
\end{split}\end{equation}
where $F_1,~F_2$ is some kernel satisfying $s^{\sigma_1} F_1(s),~s^{\sigma_2} F_2(s)$ are bounded on $]0,\infty[$
whenever $0\leq \sigma_1\leq 3/2,~0<\sigma_2\leq 3/2$, and $\xi_j^2$ is a shorthand notation for the quantity
$$\xi_j^2=\e_j^2|X_j-X'|^2(1+\e_j R_j)^{-1}(1+\e_j R')^{-1}.$$
We denote this map from $f_j$ to $U_j$ by $U_j=BS^{\e_j}[f_j]$.
We use the superscript $\e_j$ since in the new variables,
the map  depends explicitly on time through the parameter $\e_j$.

In the rest of this paper, the following notations will also be used:
\begin{equation}
R=\frac{r-r_i}{\sqrt t},\quad Z=\frac{z-z_i}{\sqrt t},\quad X=\frac{x-x_i}{\sqrt t}
\quad\mbox{and}\quad \e=\frac{\sqrt t}{r_i},
\end{equation}
here although $R,~Z,~X,~\e$ indeed depend on $i$, we omit the index $i$
for notation simplification.

After this blow-up procedure, the gaussian bound on $\omega_i$ given by \eqref{boundoti} translates into
\begin{equation}\label{boundfi}
0<f_i(t,R,Z) \leq C_{\eta,\al}e^{ - \frac{1-\eta}{4}(R^2+Z^2)},
\end{equation}
and \eqref{otiL1limit} translates into
\begin{equation}\label{fiL1limit}
\int_{\Omega_{\e}}f_i(t,R,Z)\, dRdZ\rightarrow 1,\quad\mbox{as}\quad t\rightarrow 0.
\end{equation}
We can use the estimate \eqref{boundfi} to derive the point-wise estimate for $U_i^{\e}$.
First, recalling the proof of Proposition $2.3$ in \cite{GS15}, which shows that for any $(r,z)\in\Omega$, there holds
$$|u(r,z)|\leq C\int_\Omega\frac{1}{\sqrt{(r-r')^2+(z-z')^2}}|\ot(r',z')|\,dr'dz'.$$
Then using the self-similar variables \eqref{defselfsimilarvariables},
we obtain
$$|U_i(t,R,Z)|\leq C\int_{\Omega_\e}\frac{1}{\sqrt{(R-R')^2+(Z-Z')^2}}f_i(t,R',Z')\, dR'dZ'$$
Finally substituting \eqref{boundfi} with some fixed $\eta$ into this, leads to
\begin{equation}\label{boundUi}
\bigl(1+|R|+|Z|\bigr)|U_i(t,R,Z)|\leq C_0.
\end{equation}
Using the notation \eqref{blowupotiui}, let us also do this self-similar blow-up of the whole velocity $u$
near the point $x_i\in\Omega$ and near the initial time $t=0$, and we get
\begin{equation}\label{unpackingu}
u(t,r,z)=\frac{\al_i}{\sqrt t}U_i(t,R,Z)
+\sum\limits_{j\neq i}\frac{\al_j}{\sqrt t}U_j\bigl(t,R+\frac{r_i-r_j}{\sqrt t},Z+\frac{z_i-z_j}{\sqrt t}\bigr).
\end{equation}
In view of \eqref{boundUi}, let $t\rightarrow0$ and $R,~Z$ fixed,
all $U_j\bigl(t,R+\frac{r_i-r_j}{\sqrt t},Z+\frac{z_i-z_j}{\sqrt t}\bigr)$
for $j\neq i$ vanish, and only $U_i(t,R,Z)$ remains. Thus after this blow-up procedure,
the convection term can be very close to $U_i\cdot\nabla f_i$, for a short time.
Combining with the fact that the initial measure for $\omega_i=\ot_i e_\th$ is $\al_i\delta_{x_i}$,
hence if we believe in uniqueness, it is reasonable to expect
that, for a short time, $\omega_i$ will be very close to an Oseen vortex located at $x_i$ with circulation $\al_i$.

In order to write this observation precisely, let us denote the following functions on $\R^2$:
$$w(x,y)\eqdefa e^{(|x|^2+|y|^2)/4},\quad
G(x,y)\eqdefa\frac{1}{4\pi}e^{-(|x|^2+|y|^2)/4},\quad (x,y)\in\R^2,$$
and denote by $\mathcal{X}$ the weighted space $L^2(\mathbb{R}^2, w(x,y) dxdy)$. We have:

\begin{proposition}\label{propfilimitGaussian}
{\sl For any $i\in\{1,\cdots,n\}$,
we have $\|\overline{f}_i(t,\cdot) - G(\cdot)\|_{\mathcal{X}} \to 0$ as $t$ goes to $0$,
where $\overline{f}_i$ denotes the extension of $f_i$ by zero outside $\Omega_\e$.
} \end{proposition}

 \begin{proof}
First, let us denote by $\mathcal{X}_0$ a subspace of $\mathcal{X}$, which is defined by the stronger norm
$$\|f\|_{\mathcal{X}_0}\eqdefa\|fw^{1-\eta}\|_{L^{\infty}(\mathbb{R}^2)} + \|\nabla f\|_{L^{\infty}(\mathbb{R}^2)},$$
where $\eta$ is a real number satisfying $0 < \eta < \frac 12$.
We have:
\begin{lemma}[Lemma 4.4 in \cite{GS16}]
The space $\mathcal{X}_0$ is compactly embedded in $\mathcal{X}$, and the unit ball in
$\mathcal{X}_0$ is closed for the topology induced by $\mathcal{X}$.
\end{lemma}

In the self-similar variables, the gradient bound for $\ot_i$, namely \eqref{boundnablaoti}, translates into
$$\|\nabla\overline{f}_i(t)\|_{L^{\infty}(\R^2)} < \infty,\quad\forall t\in]0,T[. $$
Combining this with the gaussian bound for $f_i$, \eqref{boundfi},
we know that, $(\overline{f}_i(t))_{0<t<T}$ is a bounded subset of $\mathcal{X}_0$, hence compact in $\mathcal{X}$.
Let $h_*$ be an accumulation point in $\mathcal{X}$ of $(\overline{f}_i(t))_{0<t<T}$ as $t$ goes to $0$,
and $(t_m)_{m\in\N}$ be the corresponding sequence of positive time satisfying
\begin{equation}\label{fitmlimit}
t_m\rightarrow 0,\quad \|\overline{f}_i(t_m) - h_*\|_{\mathcal{X}} \rightarrow 0 \quad \mbox{as}\quad m\rightarrow \infty.
\end{equation}

Now, let us temporarily consider the whole 3-D vorticity field $\omega$ and the whole 3-D velocity field $u$.
For any $m\in\N,~y\in\R^3$, and $s\in]0,t_m^{-1}T[$, we define the following sequence
 $$
 \left\{
 \begin{array}{lr}
  u^{(m)}(s,y) = \sqrt{t_m} u(t_ms, x_i + \sqrt{t_m}y) \\
  \omega^{(m)}(s,y) = t_m \omega(t_ms, x_i + \sqrt{t_m}y),
 \end{array}
 \right.
 $$
 where $x_i = (r_i,0,z_i)\in\R^3$. In other words, the vector fields $\om^{(m)},~u^{(m)}$ are defined by
 a self-similar blow-up of the original quantities $\om,~u$ near the point $x_i\in\R^3$ and near the initial time $t=0$.
 It is easy to verify that $\om,~u$ satisfy the 3-D vorticity equation:
 $$\pa_s\om^{(m)}+u^{(m)}\cdot\nabla\om^{(m)}-\Delta \om^{(m)}=
 \om^{(m)}\cdot\nabla u^{(m)},\quad\dive u^{(m)} = 0, \quad \curl u^{(m)} = \omega^{(m)},$$
for $s\in]0,t_m^{-1}T[,~y\in\R^3$.
The self-similar rescaling from $u$ to $u^{(m)}$ preserves the bounds given by \eqref{apriorimild},
precisely for all indices $k,~\ell\in\N$, we have the following {\it a priori} estimates
$$\|\pa_s^k\nabla_y^\ell u^{(m)}(s)\|_{L^\infty(\R^3)}\leq C_0 s^{-\bigl(\frac12+k+\frac{\ell}2\bigr)},
\quad s\in]0,t_m^{-1}T[,$$
which holds uniformly in $m$.
 Hence, up to an extraction, we can assume that
$$\om^{(m)}\rightarrow \overline{\omega},\quad u^{(m)}\rightarrow \overline{u},\quad\mbox{as}\quad m\rightarrow\infty,$$
with uniform convergence of both vector fields along with all their derivatives
on any compact subset of $]0,t_m^{-1}T[\times\R^3$.
Thus the limiting fields $\overline{\omega},~\overline{u}$ are smooth and satisfy
 \begin{equation}\label{overlineomegau}
  \partial_s \overline{\omega} + \overline{u} \cdot\nabla \overline{\omega}
  - \Delta \overline{\omega} = \overline{\omega} \cdot \nabla \overline{u},
  \quad\dive\overline{u}=0,\quad \curl\overline{u}=\overline{\omega}.
 \end{equation}
 The goal now is to relate $\overline{\omega}$ to $\omega_i$ and $\overline{f}_i$.
 The idea is that the other $\omega_j,~\overline{f}_j\ (j\neq i)$ should be eliminated by the blow-up procedure.
 Using the definitions, we get
 \begin{equation}\label{omegafifj}\begin{split}
\omega^{(m)}&(s,y)= t_m \omega(t_ms, x_i + \sqrt{t_m}y) \\
   &= t_m \omega(t_ms, \sqrt{(r_i + \sqrt{t_m}y_1)^2 + t_my_2^2}, 0, z_i + \sqrt{t_m}y_3)\\
   &=\Bigl( \frac {\al_i}s \overline{f}_i(t_ms,X_{ii}^{(m)}(s,y))
   +\sum_{j \neq i}\frac {\al_j}s \overline{f}_j(t_ms,X_{ij}^{(m)}(s,y))\Bigr)e_{\theta}(x_i + \sqrt{t_m}y),
 \end{split}\end{equation}
where
$$X_{ij}^{(m)}(s,y)\eqdefa \left( \frac{\sqrt{(r_i + \sqrt{t_m}y_1)^2 + t_my_2^2} - r_j}{\sqrt{t_ms}}, \frac{z_i-z_j + \sqrt{t_m}y_3}{\sqrt{t_ms}} \right).$$
If $i \neq j$, for any bounded subset $B \subset \mathbb{R}^3$ and any $y\in B$,
there exists a large constant $N_B$, such that for any $m>N_B$, there holds
$$|X_{ij}^{(m)}(s,y)|^2 \geq \frac{(r_i-r_j)^2+(z_i-z_j)^2}{2t_ms}. $$
Then the gaussian bound for $f_j$ \eqref{boundfi} entails
$$0\leq \overline{f}_j(t_ms, X_{ij}^{(m)}(s,y)) \leq C_{\eta,\al} \exp\Bigl\{- \frac{(1-\eta)|x_i-x_j|^2}{8t_ms}\Bigr\}. $$
Hence, the only contribution in the limit procedure $m \rightarrow \infty$ comes, as expected, from the $i$-th circular vortex.
Regarding $\overline{f}_i$, as shown before, $\overline{f}_i(\cdot,\cdot,t)$ is bounded in $\mathcal{X}_0$.
Thus for any fixed $s>0$, up to another extraction, there must exist some $h_s\in\mathcal{X}$ such that
\begin{equation}\label{fitmslimit}
\|\overline{f}_i(t_ms)-h_s\|_{\mathcal{X}} \rightarrow 0\quad
 \mbox{as}\quad m \rightarrow \infty.
 \end{equation}
The boundedness of $(\overline{f}_i(t_ms))_m$ in $\mathcal{X}_0$ implies that, this convergence of $(\overline{f}_i(t_ms))_m$ to $h_s$
also holds uniformly on any compact set of $\R^3$.
Therefore, taking the limit $m\rightarrow\infty$ on both sides of \eqref{omegafifj}
and noting that $e_{\theta}(x_i)=e_2=(0,1,0)$, we obtain
$$\overline{\omega}(s,y) = \frac{\al_i}s h_s\left(\frac{y_1}{\sqrt s}, \frac{y_3}{\sqrt s}\right)e_2
\eqdefa(0,\overline{\omega}_2(s,y_1,y_3),0). $$
Taking the limit $m\rightarrow\infty$ in \eqref{boundfi} and \eqref{fiL1limit}, we deduce
\begin{equation}\label{boundL1om2}
|\overline{\omega}_2(s,y_1,y_3)| \lesssim C_{\eta,\al}s^{-1} e^{- \frac{1-\eta}{4s}|y|^2},\quad
\int_{\R^2}\overline{\omega}_2(s,y_1,y_3)\, dy_1dy_3=\al_i.
\end{equation}

We now turn to the velocity field. Similarly as \eqref{omegafifj}, we can write
\begin{equation}\label{unpackingu}
u^{(m)}(s,y) = \frac{\al_i}{\sqrt s} U_i^\e(t_ms, X_{ii}^{(m)}(s,y))
+\sum_{j \neq i}\frac{\al_j}{\sqrt s} U_j^\e(t_ms, X_{ij}^{(m)}(s,y)).
\end{equation}
In view of \eqref{boundUi}, as $t_m\rightarrow0$, all $U_j^{\e}(t_ms, X_{ij}^{(m)}(s,y))$
for $j\neq i$ vanish, and only $U_i^\e(t_ms, X_{ii}^{(m)}(s,y))$ remains.
Regarding $U_i$, using \eqref{boundUi} again and taking the limit $m\rightarrow\infty$, we get
\begin{equation}\label{decayoverlineu}
|\overline{u}(s,y)| \lesssim (\sqrt s + |y_1| + |y_3|)^{-1}.
\end{equation}
Moreover, as shown in \eqref{overlineomegau}, $\overline{u}$ satisfies the following elliptic system
$$\dive\overline{u}=0,\quad \curl\overline{u}=\overline{\omega}.$$
This div-curl system has at most one solution with the decay property \eqref{decayoverlineu}, hence
$$\overline{u}(s,y)=\overline{u}_1(s,y_1,y_3)e_1+\overline{u}_3(s,y_1,y_3)e_3
=(\overline{u}_1(s,y_1,y_3),0,\overline{u}_3(s,y_1,y_3)),$$
where $(\overline{u}_1,\overline{u}_3)$ is the two dimensional velocity field obtained from
the scalar vorticity $\overline{\omega}_2$ via the Biot-Savart law in $\R^2$.

Summarizing, we have shown that the limiting vorticity $\overline{\omega}_2$,
together with the associated velocity $(\overline{u}_1,\overline{u}_3)$ solves the 2-D Navier-Stokes equations,
and it follows from \eqref{boundL1om2} that $\overline{\omega}_2(s,\cdot)$ is uniformly bounded in $L^1(\R^2)$
and converges weakly to the Dirac measure $\al_i\delta_0$ as $s\rightarrow 0$.
Then we deduce, by using Proposition $1.3$ in \cite{GallayWayne}, that
$\overline{\omega}_2(s,y_1,y_3)=\frac{\al_i}s G\left(\frac{y_1}{\sqrt s}, \frac{y_3}{\sqrt s}\right)$,
i.e. $h_s=G$ for any $s>0$. In particular, choosing $s=1$ so that $t_ms=t_m$, and comparing
\eqref{fitmlimit} with \eqref{fitmslimit}, we conclude that $h_*=G$,
which is the desired result.
\end{proof}

In view of Proposition \ref{propfilimitGaussian}, it is natural to make a further decomposition of $\omega$.
Let
$$d\eqdefa\min_{1\leq i<j\leq n}\bigl\{|x_i-x_j|,\, r_i\bigr\},$$
and $\chi:[0,\infty[\rightarrow[0,1]$ to be a smooth non-increasing cutoff function
such that $\chi=1$ on $[0,1/8]$ and $\chi$ vanishes outside $[0,1/4]$.
Let $f_0$ to be a function on $]0,T[\times\R^2$ defined as
$$f_0(t,x,y)\eqdefa G(x,y)\chi\bigl(\sqrt{t(x^2+y^2)}/d\bigr),\quad
(x,y)\in\R^2,\ t\in]0,T[,$$
and $\widetilde{f}_i$ to be a function on $]0,T[\times\Omega_\e^i$ defined as
\begin{equation}\label{decomfi}
\widetilde{f}_i(t,R,Z)=f_i(t,R,Z)-f_0(t,R,Z),\quad
(R,Z)\in\Omega_\e,\ t\in]0,T[.
\end{equation}
Then we can decompose $\ot$ further as follows:
\begin{equation}\label{decomomgaussian}
\ot(t,r,z)=\sum_{j=1}^n\Bigl(\frac{\al_j}{t}f_0(t,R_j,Z_j)
+\frac{\al_j}{t}\widetilde{f}_j(t,R_j,Z_j)\Bigr).
\end{equation}
And correspondingly, $u=BS[\ot]$ can be decomposed further into
\begin{equation}\begin{split}\label{decomugaussian}
u(t,r,z)=\sum_{j=1}^n&\Bigl(\frac{\al_j}{\sqrt t}U_{0,j}(t,R_j,Z_j)
+\frac{\al_j}{\sqrt t}\widetilde{U}_j(t,R_j,Z_j)\Bigr),\quad\mbox{where}\\
&U_{0,j}=BS^{\e_j}[f_0],\quad \widetilde{U}_j=BS^{\e_j}[\widetilde{f}_j].
\end{split}\end{equation}

\begin{rmk}\label{rmkboundaryf}
{\sl For any $j\in\{1,\cdots,n\}$, due to the cutoff function $\chi$, it is easy to see that
$f_0(t,R_j,Z_j)$ vanishes when $\sqrt t R<-d/4$, and thus vanishes when $\sqrt t R<-r_j/4$.
In particular, this implies that $f_0(t,R_j,Z_j)$ satisfies the Dirichlet boundary condition on $\pa\Omega_{\e_j}$,
and thus $\widetilde{f}_j(t,R_j,Z_j)$ also satisfies the Dirichlet boundary condition on $\pa\Omega_{\e_j}$.
}\end{rmk}

It is clear that $f_0(t)\in\cX$ for all $t\in]0,T[$, and $\|f_0(t)-G\|_\cX\rightarrow 0$ as $t\rightarrow 0$.
Thus the perturbation $\widetilde{f}_j(t)$ (extended by zero outside $\Omega_{\e_j}$)
belongs to $\cX$ for all $t\in]0,T[$, and
Proposition \ref{propfilimitGaussian} implies that $\|\widetilde{f}_j(t)\|_\cX\rightarrow 0$ as $t\rightarrow 0$.
In the next section, we shall give a more accurate quantitative rate of this convergence.

\section{Proof of Theorem \ref{thmmain}}

This section is devoted to the proof of Theorem \ref{thmmain}. In view of the decomposition \eqref{decomomgaussian},
to prove the uniqueness claim in Theorem \ref{thmmain}, we only need to show the perturbation part
$(\widetilde{f}_j)_{j=1}^n$ is uniquely determined. At the end of last section, we have shown that
$\|\widetilde{f}_j(t)\|_\cX\rightarrow 0$ as $t\rightarrow 0$, but this is not enough to prove uniqueness.
We shall give a more accurate quantitative rate of this convergence, which in particular implies the short time
estimate \eqref{ShortTimeAsymptotics}. This will be done in the first subsection.

After some modifications to the energy estimates in the proof of the short time estimate,
we can prove the uniqueness claim in Theorem \ref{thmmain}.This will be done in the second subsection.

\subsection{Short time asymptotics}

Using \eqref{eqtoti} and \eqref{blowupotiui}, we can derive the evolution equation satisfied by
the rescaled vorticity $f_i$ reads
\begin{equation}\label{eqtfi}
t \partial_t f_i(t,X) +\dive_* \bigl(\al_i U_i(t,X)f_i(t,X)+W_i(t,X)f_i(t,X)\bigr)
=(\cL f_i)(t,X)+\pa_R\Bigl(\frac{\e f_i(t,X)}{1+\e R}\Bigr),
\end{equation}
for $X\in\Omega_\e$ and $t\in]0,T[$, where the operator $\mathcal{L}$ is defined for a generic function $f$ by
$$\mathcal{L}f(X)\eqdefa\Delta_X f(X) + \frac X2 \cdot \nabla_X f(X) + f(X),$$
the operator $\dive_*$ is defined for a generic vector field $V(X)=V^r(X)e_r+V^z(X)e_z$ by
$$\dive_*\bigl(V(X)\bigr)\eqdefa \pa_R V^r(X)+\pa_Z V^z(X),$$
and $W_i$ stands for the other parts of the rescaled velocity:
$$W_i(t,X)\eqdefa\sum_{j\neq i} \al_j U_j(t,X_j),
\quad\mbox{where}\quad X_j=\frac{x-x_j}{\sqrt t}=X+\frac{x_i-x_j}{\sqrt t}.$$
Then we can deduce from \eqref{decomomgaussian},~\eqref{decomugaussian} and \eqref{eqtfi} that
\begin{equation}\label{eqttildefi}
t\partial_t \widetilde{f}_i +\al_i\dive_*(U_{0,i} \widetilde{f}_i + \widetilde{U}_i f_0
+\widetilde{U}_i \widetilde{f}_i)+\dive_* (W_if_i)
=\mathcal{L}\widetilde{f}_i +\pa_R\Bigl(\frac{\e \widetilde{f}_i}{1+\e R}\Bigr)+\mathcal{H},
\end{equation}
where
\begin{equation*}
\cH=-t\pa_t f_0+\cL f_0+\pa_R\Bigl(\frac{\e f_0(t,X)}{1+\e R}\Bigr)-\al_i\dive_*(U_{0,i}f_0).
\end{equation*}

And we shall define, following \cite{GS16}, the two types of energy for each vortex
\begin{equation}\begin{split}\label{defEj}
&E_j(t)\eqdefa\frac 12 \int_{\Omega_{\e_j}} \widetilde{f}_j(t,X_j) ^2 w(X_j)\,dX_j,\\
\mathcal{E}_j(t)\eqdefa\frac 12 \int_{\Omega_{\e_j}}& \Bigl( |\nabla\widetilde{f}_j(t,X_j)|^2 + (1+|X_j|^2) \widetilde{f}_j(t,X_j) ^2 \Bigr) w(X_j)\,dX_j,
\end{split}\end{equation}
as well as the total energies
$$E(t)\eqdefa\sum_{j=1}^n E_j(t), \qquad \mathcal{E}(t)\eqdefa\sum_{j=1}^n \mathcal{E}_j(t).$$

As we have pointed out in Remark \ref{rmkboundaryf} that, $\widetilde{f}_j$ satisfies the homogeneous Dirichlet condition on $\pa\Omega_{\e_j}$, thus although the integral in \eqref{defEj} is taken over the
time-dependent domain $\Omega_{\e_j}$, there is no contribution from the boundary when we differentiate
with respect to time. Hence we can get, by doing $L^2(\Omega_\e,w(X)dX)$ 
energy estimate to \eqref{eqttildefi} and integrating by parts, that
\begin{equation}\label{equalitydtE}
tE'_i(t)=A_i(t)+I_i(t),
\end{equation}
where
\begin{align*}
&A_i(t)=\int_{\Omega_\e}\Bigl(\mathcal{L}\widetilde{f}_i(t,X) 
+\pa_R\bigl(\frac{\e \widetilde{f}_i(t,X)}{1+\e R}\bigr)+\mathcal{H}(t,X)\\
&\qquad\qquad\qquad\qquad\qquad-\al_i\dive_*(U_{0,i} \widetilde{f}_i + \widetilde{U}_i f_0+\widetilde{U}_i \widetilde{f}_i)(t,X)\Bigr)\widetilde{f}_i(t,X)\cdot w(X)\,dX,\\
&I_i(t)=\int_{\Omega_\e}W_i(t,X)f_i(t,X)
\bigl(\nabla_X \widetilde{f}_i(t,X)+\frac X2 \widetilde{f}_i(t,X)\bigr)\cdot w(X)\,dX.
 \end{align*}

The main result of this subsection states as follows:
\begin{proposition}\label{propineqE'}
{\sl There exists some positive constant $\delta$
depending on the initial measure $\mu$, such that for $t$ sufficiently small, there holds
 \begin{equation}\label{ineqE'}
tE'_i(t)\leq-2\delta\mathcal{E}_i(t)+ C_0\sqrt t |\ln t|\mathcal{E}_i(t)^{\frac 12} +C E_i(t)^{\frac 12} \mathcal{E}_i(t)+\mathcal{R}_i(t),
\end{equation}
where the quantity $\mathcal{R}_i$ satisfies the inequality $0<\mathcal{R}_i(t)\leq e^{- C_0/t}$.
}\end{proposition}

\begin{proof}
Noting that the terms in $A_i(t)$ are exactly the same as the ones appearing on the right-hand side of
the equality $(4.42)$ in \cite{GS16}. Thus using the Proposition $4.5$ in \cite{GS16},
we know that there exists some $\e_0\in]0,1/2[$, if $t>0$ is small enough so that $\e_i<\e_0$, then
\begin{equation}\label{estimateA(t)}
A_i(t)\leq  - 2 \delta \mathcal{E}_i(t) + C\sqrt t |\ln t| E_i(t)^{\frac 12} + C E_i(t)^{\frac 12} \mathcal{E}_i(t) + \mathcal{R}_i(t).
\end{equation}
In the following we shall concentrate on the interaction part $I_i(t)$.
Using the decomposition \eqref{decomfi} and \eqref{decomugaussian}, we can write
$$W_i(t,X)f_i(t,X)=\sum\limits_{j\neq i}\bigl(\al_jU_{0,j}(t,X_j)+\al_j\widetilde{U}_j(t,X_j)\bigr)
\bigl(f_0(t,X)+\widetilde{f}_i(t,X)\bigr).$$
Thus there are four types of integral terms in $I_i(t)$, which we handle separately.

Before proceeding, let us decompose $\Omega_{\e_j}$ into two parts, namely
$$\Omega_{\e_j}^+\eqdefa\Bigl\{X \in \Omega_{\e_j} \mbox{ s.t. } |X| > \frac{d}{4\sqrt t}\Bigr\}
,\quad \Omega_{\e_j}^-\eqdefa\Bigl\{X \in \Omega_{\e_j} \mbox{ s.t. } |X| \leq \frac{d}{4\sqrt t}\Bigr\}.$$

\noindent\textbf{Type 1:} $I_{i,1}(t)=\sum\limits_{j\neq i}\int_{\Omega_\e}\al_jU_{j}(t,X_j) f_0(t,X)
\cdot\bigl(\nabla_X+X/2\bigr) \widetilde{f}_i(t,X)\cdot w(X)\,dX.$

Due to the cutoff function $\chi$, we know that $f_0(t,X)$ vanishes whenever $|X| > \frac{d}{4\sqrt t}$.
Thus $I_{i,1}(t)$ actually only integrates on $\Omega_{\e}^-$, and for $X$ in $\Omega_{\e}^-$, we have
$$|X_j|=\Bigl|X+ \frac{x_i-x_j}{\sqrt t}\Bigr|\geq\frac{3d}{4\sqrt t}.$$
Then the estimate \eqref{boundUi} gives
\begin{equation}\label{boundUjOm-}
U_{j}(t,X_j)\leq C_0 \sqrt t.
\end{equation}
Thanks to this bound, the definition of $f_0$, and Cauchy inequality, we get
\begin{equation}\begin{split}\label{estimateI1}
|I_{i,1}(t)|&\leq C_0\sqrt t\sum\limits_{j\neq i}\int_{\Omega_\e^-}e^{-|X|^2/4}
\bigl(\nabla_X \widetilde{f}_i(t,X)+\frac X2 \widetilde{f}_i(t,X)\bigr)w(X)\,dX\\
&\leq C_0 \sqrt t \bigl\|e^{-|X|^2/8}\bigr\|_{L^2(\Omega_\e^-)}
\bigl\|\bigl(\nabla_X+X/2\bigr) \widetilde{f}_i(t,X)\cdot w(X)^{1/2}\bigr\|_{L^2(\Omega_\e^-)}\\
&\leq C_0 \sqrt t \mathcal{E}_i(t)^{\frac 12}.
\end{split}\end{equation}

\noindent\textbf{Type 2:} $I_{i,2}(t)=\sum\limits_{j\neq i}\int_{\Omega_\e} \al_jU_{j}(t,X_j)\widetilde{f}_i(t,X)
\cdot \bigl(\nabla_X+X/2\bigr) \widetilde{f}_i(t,X)\cdot w(X)\,dX.$

We decompose $I_{i,2}$ into two different parts according to the integra domain.
On $\Omega_\e^-$, by using the bound \eqref{boundUjOm-} and Cauchy inequality again, we obtain
\begin{equation}\label{estimateI2-}
\Bigl|\int_{\Omega_\e^-}U_{j}(t,X_j) \widetilde{f}_i(t,X)
\cdot \bigl(\nabla_X+X/2\bigr) \widetilde{f}_i(t,X)\cdot w(X)\, dX \Bigr|
\leq C_0 \sqrt t E_i(t)^{\frac 12} \mathcal{E}_i(t)^{\frac 12}.
 \end{equation}

To handle the integral on $\Omega_\e^+$, a mere application of \eqref{boundUi} gives
\begin{equation}\label{boundU0i}
\|U_{j}\|_{L^{\infty}_T(L^\infty(\Omega_{\e_j}))}\leq C_0.
\end{equation}
And it follows from the Gaussian bound for $f_i$ \eqref{boundfi} and the fact that $f_0$ vanishes
on $\Omega_\e^+$ that, the same Gaussian bound also holds for $\widetilde{f}_i$, precisely
\begin{equation}\label{pointboundfitilde}
0<\widetilde{f}_i(t,X) \leq C_{\eta,\al}e^{ - \frac{1-\eta}{4}|X|^2},\quad\forall X\in\Omega_\e^+.
\end{equation}
Using the above bounds \eqref{boundU0i} and \eqref{pointboundfitilde} with $\eta=\frac14$, we get
\begin{align*}
\Bigl|\int_{\Omega_\e^+}U_{j}(t,X_j) \widetilde{f}_i(t,X)
\cdot \bigl(\nabla_X+X/2\bigr) \widetilde{f}_i(t,X)\cdot w(X)\, dX \Bigr| 
& \leq C_0\|\widetilde{f}_i(t)w^{\frac 12}\|_{L^2(\Omega_\e^+)} \mathcal{E}_i(t)^{\frac 12}\\
&\leq C_0e^{-\frac{d^2}{256t}}\mathcal{E}_i(t)^{\frac 12}.
\end{align*}
Combining this with the estimate \eqref{estimateI2-}, we finally get
\begin{equation}\label{estimateI2}
|I_{i,2}(t)|\leq  C_0 \sqrt t E_i(t)^{\frac 12} \mathcal{E}_i(t)^{\frac 12}
+C_0e^{-\frac{d^2}{256t}}\mathcal{E}_i(t)^{\frac 12}.
\end{equation}

Substituting the estimates \eqref{estimateA(t)},~\eqref{estimateI1} and\eqref{estimateI2}
,and using the trivial bounds
$$E_i \leq \mathcal{E}_i \leq \mathcal{E}, \quad  E_i \leq E$$
allows us to obtain
\begin{align*}
tE'_i(t) \leq  - 2 \delta \mathcal{E}_i(t) +& C\sqrt t |\ln t| E_i(t)^{\frac 12} + C E_i(t)^{\frac 12} \mathcal{E}_i(t) + \mathcal{R}_i(t)\\
&+C_0\sqrt t \mathcal{E}_i(t)^{\frac 12}
+ C_0 \sqrt t E_i(t)^{\frac 12} \mathcal{E}_i(t)^{\frac 12}
+C_0e^{-\frac{d^2}{256t}}\mathcal{E}_i(t)^{\frac 12}.
\end{align*}
Recalling that $E(t)$ goes to $0$ as $t$ goes to $0$ yields the simplified bound
$$tE'_i(t)\leq-2\delta\mathcal{E}_i(t)+ C_0\sqrt t |\ln t|\mathcal{E}_i(t)^{\frac 12} +C E_i(t)^{\frac 12} \mathcal{E}_i(t)+\mathcal{R}_i(t),$$
which is the desired differential inequality. This completes the proof of this proposition.
\end{proof}

\begin{proof}[Proof of the estimate \eqref{ShortTimeAsymptotics}]
Applying Young's inequality to \eqref{ineqE'} gives
\begin{equation}\label{tE'1}
tE'_i(t)\leq-\frac32 \delta\mathcal{E}_i(t)+ C_0 t |\ln t|^2+C E_i(t)^{\frac 12} \mathcal{E}_i(t)+\mathcal{R}_i(t).
\end{equation}
Recalling that by definition $\e_i=\sqrt t/r_i$ and $E(t)$ goes to $0$ as $t$ goes to $0$, 
thus there exists some small constant
$t_0$ depending only on the initial measure $\mu$, such that
both $\e_i<\e_0$ and $E_i(t)^{1/2}<\delta/2$ hold whenever $t<t_0$.
Combining this with the facts that $E_i\leq\mathcal{E}_i$ and $0<\mathcal{R}_i(t)\leq e^{- C_0/t}$,
we can get from \eqref{tE'1}, for $t<t_0$, that
\begin{align*}
tE'_i(t)&\leq-\delta\mathcal{E}_i(t)+ C_0 t |\ln t|^2+\mathcal{R}_i(t)\\
&\leq-\delta E_i(t)+ C_0 t |\ln t|^2.
\end{align*}
Integrating this differential inequality yields the bound
\begin{equation}\label{rateEi}
E_i(t)\leq C_0 t^{-\delta}\int_0^t s^\delta|\ln s|^2\,ds
\leq C_0 t|\ln t|^2.
\end{equation}
Then in view of the definition \eqref{defEj},
the above inequality leads to
$$\|f_i(t)-f_0(t)\|_{L^1(\Omega_\e)}=\|\widetilde{f}_i\|_{L^1(\Omega_\e)}
\leq C E_i^{1/2}(t)\leq C_0 \sqrt t|\ln t|.$$
And since $f_0$ is extremely close to $G$, we finally obtain
\begin{equation}\begin{split}\label{fi-G}
\|f_i(t)-G\|_{L^1(\Omega_\e)}&\leq
\|f_i(t)-f_0(t)\|_{L^1(\Omega_\e)}+\|f_0(t)-G\|_{L^1(\Omega_\e)}\\
&\leq C_0\sqrt t|\ln t|+e^{-C_0/t}\leq C_0\sqrt t|\ln t|.
\end{split}\end{equation}
Returning to the original variables, and summing up over $i$, gives exactly the short time estimate
\eqref{ShortTimeAsymptotics} for $t<t_0$.
\end{proof}

\subsection{Uniqueness}

The purpose of this final subsection is to prove the uniqueness result in Theorem \ref{thmmain}.
Assume that $\oto,~\ott\in\cC\bigl(]0,T[, L^1(\Omega) \cap L^{\infty}(\Omega)\bigr)$
are two mild solutions to the vorticity equation \eqref{NSom} satisfying \eqref{thmcondi}.
Introducing the self-similar variables and decompose these two solutions just as what we have done
in Subsection $2.2$, precisely for $\ell=1,2$, we write
$$\otl(t,r,z)=\sum_{j=1}^n\frac{\al_j}{t}f_j^{(\ell)}(t,R_j,Z_j)
=\sum_{j=1}^n\Bigl(\frac{\al_j}{t}f_0(t,R_j,Z_j)
+\frac{\al_j}{t}\widetilde{f}_j^{(\ell)}(t,R_j,Z_j)\Bigr),$$
and correspondingly, $u^{(\ell)}=BS[\otl]$ can be decomposed into
$$u(t,r,z)^{(\ell)}=\sum_{j=1}^n \frac{\al_j}{\sqrt t}U_j^{(\ell)}(t,R_j,Z_j)
=\sum_{j=1}^n\Bigl(\frac{\al_j}{\sqrt t}U_{0,j}(t,R_j,Z_j)
+\frac{\al_j}{\sqrt t}\widetilde{U}_j^{(\ell)}(t,R_j,Z_j)\Bigr).$$
The differences of the rescaled solutions will be denoted by
$$\widetilde{f}_i^\Delta \eqdefa f_i^{(1)} - f_i^{(2)} = \widetilde{f}_i^{(1)} - \widetilde{f}_i^{(2)},
\quad\widetilde{U}_i^\Delta \eqdefa U_i^{(1)} - U_i^{(2)} = \widetilde{U}_i^{(1)} - \widetilde{U}_i^{(2)}.$$
The evolution equation for $\widetilde{f}_i^\Delta$ reads
\begin{equation}\begin{split}\label{eqtdifferencefi}
t\partial_t &\widetilde{f}_i^\Delta+\al_i\dive_*(U_{0,i} \widetilde{f}_i^\Delta + \widetilde{U}_i^\Delta f_0)
+\al_i\dive_*(\widetilde{U}_i^{(1)} \widetilde{f}_i^{(1)}-\widetilde{U}_i^{(2)} \widetilde{f}_i^{(2)})\\
&+\dive_*(W_{0,i} \widetilde{f}_i^\Delta + \widetilde{W}_i^\Delta f_0)
+\dive_*(\widetilde{W}_i^{(1)} \widetilde{f}_i^{(1)}-\widetilde{W}_i^{(2)} \widetilde{f}_i^{(2)})
=\mathcal{L}\widetilde{f}_i^\Delta +\pa_R\Bigl(\frac{\e \widetilde{f}_i^\Delta}{1+\e R}\Bigr),
\end{split}\end{equation}
where
$$W_{0,i}(t,X)\eqdefa\sum_{j\neq i} \al_j U_{0,j}(t,X_j),\quad
\widetilde{W}_i^{(\ell)}(t,X)\eqdefa\sum_{j\neq i} \al_j \widetilde{U}_j^{(\ell)}(t,X_j).$$
In analogy with \eqref{defEj}, the energies for each solution are straightforwardly denoted by
$$E_j^{(\ell)}(t)\eqdefa\frac 12 \int_{\Omega_{\e_j}} \widetilde{f}^{(\ell)}_j(t,X_j) ^2 w(X_j)\,dX_j,
\quad E^{(\ell)}(t)\eqdefa\sum_{j=1}^n E_j^{(\ell)}(t),$$
$$\mathcal{E}_j^{(\ell)}(t)\eqdefa\frac 12 \int_{\Omega_{\e_j}}
\Bigl( |\nabla\widetilde{f}_j^{(\ell)}(t,X_j)|^2
+(1+|X_j|^2)\widetilde{f}_j^{(\ell)}(t,X_j) ^2 \Bigr) w(X_j)\,dX_j,
\  \mathcal{E}^{(\ell)}(t)\eqdefa\sum_{j=1}^n \mathcal{E}_j^{(\ell)}(t),$$
as well as the energies for the difference 
$$E_j^\Delta(t)\eqdefa\frac 12 \int_{\Omega_{\e_j}} \widetilde{f}_j^\Delta(t,X_j) ^2 w(X_j)\,dX_j,
\quad E^\Delta(t)\eqdefa\sum_{j=1}^n E_j^\Delta(t),$$
$$\mathcal{E}_j^\Delta(t)\eqdefa\frac 12 \int_{\Omega_{\e_j}}
\Bigl( |\nabla\widetilde{f}_j^\Delta(t,X_j)|^2 
+(1+|X_j|^2)\widetilde{f}_j^\Delta(t,X_j) ^2 \Bigr) w(X_j)\,dX_j,
\quad \mathcal{E}^\Delta(t)\eqdefa\sum_{j=1}^n \mathcal{E}_j^\Delta(t).$$
In view of \eqref{rateEi}, combining with the elementary fact that $E_j^\Delta\leq
2\bigl(E_j^{(1)}+E_j^{(2)}\bigr)$, we know that $E_j^\Delta(t)$ also decays to $0$ with rate 
at least $t|\ln t|^2$ as $t\rightarrow 0$. We believe that $E_j^\Delta(t)$ decays faster than $E_j^{(\ell)}$
since the source $\cH$ and $\dive_*(W_{0,i}f_0)$ has disappeared when taking the difference of the equations
for $ f_i^{(1)}$ and $f_i^{(2)}$. Precisely, we have:

\begin{proposition}\label{propEdifference}
{\sl There exists a positive time $t_1$ such that for all $0<t<t_1$, there holds
\begin{equation}\label{estimateEDelta}
E^\Delta(t) \leq e^{-C_0/t}.
\end{equation}
}\end{proposition}

\begin{proof}
Similarly as in the proof of Proposition \ref{propineqE'}, by doing an $L^2(\Omega_\e,w(X)dX)$
energy estimate to \eqref{eqtdifferencefi} and integrating by parts, we obtain
\begin{equation}\label{equalitydtEi}
t\frac{d}{dt}E^\Delta_i(t)=A^\Delta_i(t)+I_i^\Delta(t),
\end{equation}
where
\begin{align*}
&A_i^\Delta(t)=\int_{\Omega_\e}\Bigl(\mathcal{L}\widetilde{f}_i^\Delta(t,X)
+\pa_R\bigl(\frac{\e \widetilde{f}^\Delta_i(t,X)}{1+\e R}\bigr)
-\al_i\dive_*(U_{0,i} \widetilde{f}_i^\Delta + \widetilde{U}_i^\Delta f_0)\\
&\qquad\qquad\qquad\qquad\qquad
-\al_i\dive_*(\widetilde{U}_i^{(1)} \widetilde{f}_i^{(1)}-\widetilde{U}_i^{(2)} \widetilde{f}_i^{(2)})
\Bigr)\widetilde{f}_i^\Delta(t,X)\cdot w(X)\,dX,\\
&I_i^\Delta(t)=\int_{\Omega_\e}\bigl(W_{0,i} \widetilde{f}_i^\Delta + \widetilde{W}_i^\Delta f_0
+\widetilde{W}_i^{(1)} \widetilde{f}_i^{(1)}-\widetilde{W}_i^{(2)} \widetilde{f}_i^{(2)}\bigr)(t,X)\cdot
\bigl(\nabla_X+X/2\bigr)\widetilde{f}_i^\Delta(t,X)\cdot w(X)\,dX.
\end{align*}

First, the estimate $(4.71)$ of \cite{GS16} claims that there exists some positive constant $\delta$
and some $\e_0\in]0,1[$ such that as long as $\e<\e_0$, there holds
\begin{equation}\label{ineqADelta}
A_i^\Delta(t)\leq-2\delta\mathcal{E}_i^\Delta(t)+C\bigl(E_i^{(1)}(t)^{\frac 12}+E_i^{(2)}(t)^{\frac 12}\bigr) 
\mathcal{E}_i^\Delta(t)+\mathcal{R}_i^\Delta(t),
\end{equation}
where the quantity $\mathcal{R}^\Delta_i$ satisfies the inequality $0<\mathcal{R}^\Delta_i(t)\leq e^{- C_0/t}$.
We mention that the terms with type $C_0\sqrt t |\ln t|\mathcal{E}_i(t)^{\frac 12}$
in \eqref{estimateA(t)} does not appear here, due to the cancellation of the source term $\cH$ when taking
the difference.

For the interaction part $I_i^\Delta(t)$, thanks to the cancellation of $\dive_*(W_{0,i}f_0)$,
there are only three types of integral terms, which we handle separately in the following.

\noindent\textbf{Type 1:} $I_{i,1}^\Delta(t)=\int_{\Omega_\e} W_{0,i}(t,X) \widetilde{f}_i^\Delta(t,X)\cdot
\bigl(\nabla_X+X/2\bigr)\widetilde{f}_i^\Delta(t,X)\cdot w(X)\,dX.$

We decompose $I_{i,1}^\Delta$ into two different parts according to the integra domain.
On $\Omega_\e^-$, we have the point-wise estimate:
\begin{lem} \label{lempointboundU0j}
{\sl For any $j\neq i$, and any $X_j$ in $\Omega_{\e_j}^-$ (i.e. $X$ in $\Omega_\e^-$), we have
 $$|U_{0,j}(t,X_j)|  \leq C_0 \sqrt t. $$
}\end{lem}

\begin{proof}
Using the explicit formula \eqref{paraBS}, and the fact that
$f_0$ supports in $\Omega_\e^-$, we get
\begin{align*}
&U_{0,j}^r(t,X_j)
=\frac{1}{2\pi}\int_{\Omega_\e^-}\sqrt{(1+\e_j R')(1+\e_j R_j)^{-1}}
F_1(\xi_j^2)\frac{Z_j-Z'}{|X_j-X'|^2} f_0(t,X')\,dX',\\
&U_{0,j}^z(t,X_j)
=-\frac{1}{2\pi}\int_{\Omega_\e^-}\sqrt{(1+\e_j R')(1+\e_j R_j)^{-1}}
F_1(\xi_j^2)\frac{R_j-R'}{|X_j-X'|^2} f_0(t,X')\,dX'\\
&\qquad\qquad\quad+\frac{\e_j}{4\pi}\int_{\Omega_\e^- }\sqrt{(1+\e_j R')(1+\e_j R_j)^{-3}}
\bigl(F_1(\xi_j^2)+F_2(\xi_j^2)\bigr) f_0(t,X')\,dX',
\end{align*}
where
$$\xi_j^2=\e_j^2|X_j-X'|^2(1+\e_j R_j)^{-1}(1+\e_j R')^{-1}.$$
For $X$ and $X'$ in $\Omega_\e^-$, we have
$$|X_j-X'|=\Bigl|X-X' + \frac{x_i-x_j}{\sqrt t}\Bigr| \in\bigl[\frac{d}{2\sqrt t},
\frac{1}{\sqrt t}\bigl(|x_i-x_j|+\frac d2\bigr)\bigr],$$
$$1+\e_j R'\in\bigl[\frac34,\frac54\bigr],\quad \mbox{and}\quad
1+\e_j R_j=\frac{r_i}{r_j}+\frac{\sqrt t R}{r_j}\in\bigl[\frac{3r_i}{4r_j},\frac{5r_i}{4r_j}\bigr].$$
Using the above bounds and the fact that $F_1(s),~s^{\frac12} F_2(s)$ are bounded on $]0,\infty[$,
we achieve
\begin{equation*}
|U_{0,j}(X_j)|\leq C_0\int_{\Omega_\e^-}\sqrt t e^{-|X'|^2/4}\,dX'\leq C_0\sqrt t,
\end{equation*}
which completes the proof of this lemma.
\end{proof}

A direct consequence of this lemma is that, $W_{0,i}(t,X)\leq C_0 \sqrt t$ for any $X\in\Omega_\e^-$.
Using this point-wise bound and Cauchy inequality, we obtain
\begin{equation}\label{estimateIiDelta-}
\Bigl|\int_{\Omega_\e^-} W_{0,i}(t,X) \widetilde{f}_i^\Delta(t,X)\cdot
\bigl(\nabla_X+X/2\bigr)\widetilde{f}_i^\Delta(t,X)\cdot w(X)\,dX\Bigr|
\leq C_0 \sqrt t E_i^\Delta(t)^{\frac 12} \mathcal{E}_i^\Delta(t)^{\frac 12}.
 \end{equation}

To handle the integral on $\Omega_\e^+$, we need some more careful estimates on the rescaled velocity.
After the blow-up procedure \eqref{blowupotiui},
Proposition $2.3$ of \cite{GS15} translates into:
\begin{lem}
{\sl i) If $1<p<2<q<\infty,~\frac1q=\frac1p-\frac12$, then
\begin{equation}\label{GS15 prop2.3i}
\|BS^\e[f]\|_{L^q(\Omega_\e)}\leq C \|f\|_{L^p(\Omega_\e)}.
\end{equation}
ii) If $1\leq p<2<q\leq \infty$, then
\begin{equation}\label{GS15 prop2.3ii}
\|BS^\e[f]\|_{L^\infty(\Omega_\e)}\leq C \|f\|_{L^p(\Omega_\e)}^\sigma\|f\|_{L^q(\Omega_\e)}^{1-\sigma},
\quad\mbox{where}\quad \sigma=\frac p2\frac{q-2}{q-p}\in]0,1[.
\end{equation}
}\end{lem}

It follows from a mere application of \eqref{GS15 prop2.3ii} to a gaussian function that
\begin{equation}\label{boundW0i}
\|W_{0,i}\|_{L^{\infty}_T(L^\infty(\Omega_{\e}))}\leq C.
\end{equation}
And it follows from the Gaussian bound for $f_i^{(\ell)}$ \eqref{boundfi} and the fact that $f_0$ vanishes
on $\Omega_\e^+$ that, the same Gaussian bound also holds for $\widetilde{f}_i^{(\ell)}$, precisely
\begin{equation}\label{pointboundfitilde}
0<\widetilde{f}_i^{(\ell)}(t,X) \leq C_{\eta,\al}e^{ - \frac{1-\eta}{4}|X|^2},\quad\forall X\in\Omega_\e^+.
\end{equation}
Using the above bounds \eqref{boundW0i} and \eqref{pointboundfitilde} with $\eta=\frac14$, we get
\begin{align*}
\Bigl|\int_{\Omega_\e^+}W_{0,i}(t,X) \widetilde{f}_i^\Delta(t,X)\cdot
\bigl(\nabla_X+X/2\bigr)\widetilde{f}_i^\Delta(t,X)\cdot w(X)\,dX\Bigr|
& \leq C\|\widetilde{f}_i^\Delta(t)w^{\frac 12}\|_{L^2(\Omega_\e^+)} \mathcal{E}_i^\Delta(t)^{\frac 12}\\
&\leq C_0 e^{-\frac{d^2}{256t}}\mathcal{E}_i^\Delta(t)^{\frac 12}.
\end{align*}
Combining this with the estimate \eqref{estimateIiDelta-}, we finally get
\begin{equation}\label{estimateIiDelta}
|I_{i,1}^\Delta(t)|\leq  C_0 \sqrt t E_i^\Delta(t)^{\frac 12} \mathcal{E}_i^\Delta(t)^{\frac 12}
+C_0 e^{-\frac{d^2}{256t}}\mathcal{E}_i^\Delta(t)^{\frac 12}.
\end{equation}

\noindent\textbf{Type 2:} $I_{i,2}^\Delta(t)=\int_{\Omega_\e} \widetilde{W}_i^\Delta(t,X) f_0(t,X)\cdot
\bigl(\nabla_X+X/2\bigr)\widetilde{f}_i^\Delta(t,X)\cdot w(X)\,dX.$

Noting that $f_0$ supports only on $\Omega_\e^-$, and $f_0(X)w(X)\leq 1$ on $\Omega_\e$, we get
\begin{equation}\label{boundIIiDelta1}
|I_{i,2}^\Delta(t)|\leq\int_{\Omega_\e^-}\sum\limits_{j\neq i} 
\bigl|\al_j(\widetilde{U}_j^{(1)}-\widetilde{U}_j^{(2)})(t,X_j)\cdot
\bigl(\nabla_X+X/2\bigr)\widetilde{f}_i^\Delta(t,X)\bigr| \,dX.
\end{equation}

Let us decompose $\widetilde{U}_j^{(\ell)}$ as the sum of $\widetilde{U}_j^{(\ell),+}$
and $\widetilde{U}_j^{(\ell),-}$, with
$$\widetilde{U}_j^{(\ell),\pm}(X_j)
\eqdefa BS^{\e_j}\bigl[\widetilde{f}_j^{(\ell)}(X_j)
\mathbf{1}_{\Omega^\pm_{\e_j}}(X_j)\bigr],$$
where $\mathbf{1}_{\Omega^\pm_{\e}}$ stands for the characteristic function of $\Omega^\pm_{\e}$.

Exactly along the proof of Lemma \ref{lempointboundU0j}, we can get, for any $X\in\Omega_\e^-$, that
\begin{align*}
\bigl|\bigl(\widetilde{U}_j^{(1),-}-\widetilde{U}_j^{(2),-}\bigr)\bigl(X+\frac{x_i-x_j}{\sqrt t}\bigr)\bigr|
&\leq C_0\sqrt t\int_{\Omega_{\e_j}^-} \bigl|\widetilde{f}_j^{(1)}(X')-\widetilde{f}_j^{(2)}(X')\bigr|\,dX'\\
&\leq C_0\sqrt t \|w^{-1/2}\|_{L^2} E_j^\Delta(t)^{\frac 12}\\
&\leq C_0\sqrt t E_j^\Delta(t)^{\frac 12}.
\end{align*}
Using this bound and the fact that $L^2\bigl(\Omega_\e^-,w(X)dX\bigr)\hookrightarrow L^1(\Omega_\e^-,dX)$
, we achieve
\begin{equation}\begin{split}\label{boundIIiDelta1-}
\int_{\Omega_\e^-}\sum\limits_{j\neq i}&
\bigl|\al_j(\widetilde{U}_j^{(1),-}-\widetilde{U}_j^{(2),-})(t,X_j)\cdot
\bigl(\nabla_X+X/2\bigr)\widetilde{f}_i^\Delta(t,X)\bigr| \,dX\\
&\leq C_0\sqrt t E^\Delta(t)^{\frac 12}\mathcal{E}_i^\Delta(t)^{\frac 12}.
\end{split}\end{equation}

For $\widetilde{U}_j^{(\ell),+}$, we use \eqref{GS15 prop2.3i} with $p=4/3,~q=4$,
and H\"older's inequality to obtain
\begin{align*}
\bigl\|\widetilde{U}_j^{(1),+}-\widetilde{U}_j^{(2),+}\bigr\|_{L^4(\Omega_{\e_j})}
&\leq C_0\bigl\|\widetilde{f}_j^{(1)}-\widetilde{f}_j^{(2)}\bigr\|_{L^{\frac43}(\Omega_{\e_j}^+)}\\
&\leq C_0\|w^{-1/2}\|_{L^4(\Omega_{\e_j}^+)}
\bigl\|\bigl(\widetilde{f}_j^{(1)}-\widetilde{f}_j^{(2)}\bigr)w^{1/2}\bigr\|_{L^2(\Omega_{\e_j}^+)}\\
&\leq C_0e^{-C_0/t} E_j^\Delta(t)^{\frac 12}.
\end{align*}
Using this estimate and H\"older's inequality again, we achieve
\begin{equation}\begin{split}\label{boundIIiDelta1+}
\int_{\Omega_\e^-}\sum\limits_{j\neq i}&
\bigl|\al_j(\widetilde{U}_j^{(1),+}-\widetilde{U}_j^{(2),+})(t,X_j)\cdot
\bigl(\nabla_X+X/2\bigr)\widetilde{f}_i^\Delta(t,X)\bigr| \,dX\\
&\leq \sum\limits_{j\neq i}\bigl\|\widetilde{U}_j^{(1),+}-\widetilde{U}_j^{(2),+}\bigr\|_{L^4(\Omega_\e^-)}
\|w^{-1/2}\|_{L^4(\Omega_\e^-)}\bigl\|\bigl(\nabla_X+X/2\bigr)\widetilde{f}_i^\Delta
\cdot w^{1/2}\bigr\|_{L^2(\Omega_\e^-)}\\
&\leq C_0e^{-C_0/t} E^\Delta(t)^{\frac 12}\mathcal{E}_i^\Delta(t)^{\frac 12}.
\end{split}\end{equation}

Combining the estimates \eqref{boundIIiDelta1-} and \eqref{boundIIiDelta1+}, we finally achieve that
\begin{equation}\label{estimateIIiDelta}
|I_{i,2}^\Delta(t)|\leq C_0\sqrt t E^\Delta(t)^{\frac 12}\mathcal{E}_i^\Delta(t)^{\frac 12}. 
\end{equation}

\noindent\textbf{Type 3:} $I_{i,3}^\Delta(t)=\int_{\Omega_\e} \bigl(\widetilde{W}_i^{(1)} \widetilde{f}_i^{(1)}-\widetilde{W}_i^{(2)} \widetilde{f}_i^{(2)}\bigr)(t,X)\cdot
\bigl(\nabla_X+X/2\bigr)\widetilde{f}_i^\Delta(t,X)\cdot w(X)\,dX.$

The strategy of estimating $I_{i,3}^\Delta(t)$ is to write
$$\widetilde{W}_i^{(1)} \widetilde{f}_i^{(1)}-\widetilde{W}_i^{(2)} \widetilde{f}_i^{(2)}
= \widetilde{W}_i^\Delta \widetilde{f}_i^{(1)}+\widetilde{W}_i^{(2)} \widetilde{f}_i^\Delta,$$
where $\widetilde{W}_i^\Delta\eqdefa \widetilde{W}_i^{(1)}-\widetilde{W}_i^{(2)}$.
Then we get, by using H\"older's inequality, that
\begin{equation}\begin{split}\label{estimateIIIiDelta1}
|I_{i,3}^\Delta(t)|&\leq\bigl(\bigl\|\widetilde{W}_i^\Delta \bigr\|_{L^\infty(\Omega_\e)}
\bigl\|\widetilde{f}_i^{(1)}w^{\frac12} \bigr\|_{L^2(\Omega_\e)}
+\bigl\|\widetilde{W}_i^{(2)}\bigr\|_{L^\infty(\Omega_\e)}
\bigl\|\widetilde{f}_i^\Delta w^{\frac12} \bigr\|_{L^2(\Omega_\e)}\bigr)\\
&\qquad\qquad\qquad\qquad\qquad\qquad\qquad\qquad
\times\bigl\|\bigl(\nabla_X+X/2\bigr)\widetilde{f}_i^\Delta w^{\frac12}\bigr\|_{L^2(\Omega_\e)}\\
&\leq\bigl(\bigl\|\widetilde{W}_i^\Delta \bigr\|_{L^\infty(\Omega_\e)}E_i^{(1)}(t)^{\frac12}
+\bigl\|\widetilde{W}_i^{(2)}\bigr\|_{L^\infty(\Omega_\e)}E_i^\Delta(t)^{\frac 12}\bigr)
\mathcal{E}_i^\Delta(t)^{\frac 12}.
\end{split}\end{equation}
By using \eqref{GS15 prop2.3ii} with $p=4/3,~q=4$, and Gagliardo-Nirenberg inequality, we obtain
\begin{align*}
\bigl\|\widetilde{W}_i^\Delta \bigr\|_{L^\infty(\Omega_\e)}
&\leq C_0\sum\limits_{j\neq i}\bigl\|\widetilde{f}_j^\Delta\bigr\|_{L^{4/3}(\Omega_\e)}^{1/2}
\bigl\|\widetilde{f}_j^\Delta\bigr\|_{L^{4}(\Omega_\e)}^{1/2}\\
&\leq C_0\sum\limits_{j\neq i}\bigl\|\widetilde{f}_j^\Delta w^{1/2}\bigr\|_{L^{2}(\Omega_\e)}^{1/2}
\bigl\|w^{-1/2}\bigr\|_{L^{2}(\Omega_\e)}^{1/2}
\bigl\|\widetilde{f}_j^\Delta\bigr\|_{L^2(\Omega_\e)}^{1/4}
\bigl\|\nabla\widetilde{f}_j^\Delta\bigr\|_{L^2(\Omega_\e)}^{1/4}\\
&\leq C_0\sum\limits_{j\neq i}E_j^\Delta(t)^{\frac 38}
\mathcal{E}_j^\Delta(t)^{\frac 18}.
\end{align*}
Similarly, and noting that $\widetilde{f}_j^{(2)}$ satisfies the point-wise estimate
\eqref{pointboundfitilde}, we obtain
\begin{align*}
\bigl\|\widetilde{W}_i^{(2)}\bigr\|_{L^\infty(\Omega_\e)}
&\leq C_0\sum\limits_{j\neq i}\bigl\|\widetilde{f}_j^{(2)}\bigr\|_{L^{4/3}(\Omega_\e)}^{1/2}
\bigl\|\widetilde{f}_j^{(2)}\bigr\|_{L^{4}(\Omega_\e)}^{1/2}\\
&\leq C_0\sum\limits_{j\neq i}E_j^{(2)}(t)^{\frac14}.
\end{align*}
Substituting the above two estimates into \eqref{estimateIIIiDelta1}, we achieve
\begin{equation}\begin{split}\label{estimateIIIiDelta}
|I_{i,3}^\Delta(t)|\leq C_0\Bigl(E_i^{(1)}(t)^{\frac12}E^\Delta(t)^{\frac 38}\mathcal{E}^\Delta(t)^{\frac 18}
+E^{(2)}(t)^{\frac 14}E_i^\Delta(t)^{\frac 12}\Bigr)
\mathcal{E}_i^\Delta(t)^{\frac 12}.
\end{split}\end{equation}

Overall, by putting \eqref{estimateIiDelta},~\eqref{estimateIIiDelta} and \eqref{estimateIIIiDelta}
together, using Young's inequality and the fact that
 $E_i^\Delta\leq \mathcal{E}_i^\Delta\leq \mathcal{E}^\Delta$, we achieve
\begin{equation}\label{ineqIDelta}
I^\Delta(t)\leq \delta\mathcal{E}_i^\Delta(t)
+C_0\bigl(\sqrt t
+E_i^{(1)}(t)^{\frac12}+E^{(2)}(t)^{\frac 14}\bigr)\mathcal{E}^\Delta(t)
+C_0e^{-C_0/t}.
\end{equation}
Then substituting \eqref{ineqADelta} and \eqref{ineqIDelta} into \eqref{equalitydtEi},
and summing up over $i$, leads to
\begin{equation}\label{equalitydtE}
t\frac{d}{dt}E^\Delta(t)\leq -\delta\mathcal{E}^\Delta(t)+C_0\bigl(\sqrt t
+E^{(1)}(t)^{\frac12}+E^{(2)}(t)^{\frac12}+E^{(2)}(t)^{\frac 14}\bigr)\mathcal{E}^\Delta(t)
+C_0e^{-C_0/t}.
\end{equation}
The bound \eqref{rateEi} guarantees the existence of a positive time $t_1$, such that
for all $0<t<t_1$, there holds
$C_0\bigl(\sqrt t+E^{(1)}(t)^{\frac12}+E^{(2)}(t)^{\frac12}+E^{(2)}(t)^{\frac 14}\bigr)\leq \frac{\delta}2.$
Then \eqref{equalitydtE} turns into
\begin{equation}\label{ineqdtEDelta}
t\frac{d}{dt}E^\Delta(t)\leq -\frac\delta{2}\mathcal{E}^\Delta(t)+C_0e^{-C_0/t}
\leq -\frac\delta{2} E^\Delta(t)+C_0e^{-C_0/t}.
\end{equation}
Then integrating this differential inequality from $0$ to $t<t_1$ gives
$$E^\Delta(t)\leq C_0 t^{-\delta/2}\int_0^t s^{\delta/2-1}e^{-C_0/s}\,ds
\leq e^{-C_0/t},$$
which is exactly the desired estimate \eqref{estimateEDelta}.
\end{proof}

Proposition \ref{propEdifference} already shows that $E^\Delta(t)$ converges extremely rapidly to $0$
as $t\rightarrow0$, but our actual goal is to prove that $E^\Delta(t)$ vanishes identically,
which will be done in the following.

\begin{proof}[Proof of the uniqueness result in Theorem \ref{thmmain}]
The key is to get a new differential inequality for $E^\Delta(t)$ like \eqref{ineqdtEDelta},
but in which the ``inhomogeneous"
term like $C_0e^{-C_0/t}$ does not appear.

First, the estimate $(4.73)$ of \cite{GS16} claims that as long as $\e<1/2$, there holds
\begin{equation}\label{ineqAhomo}
A_i^\Delta(t)\leq-\delta\mathcal{E}_i^\Delta(t)+C_0E_i^\Delta(t)
+C_0 \bigl(E_i^{(1)}(t)^{\frac 12}+E_i^{(2)}(t)^{\frac 12}\bigr)
\mathcal{E}_i^\Delta(t).
\end{equation}

For the estimate of $I_i^\Delta(t)$, we only need to modify the estimate of $I_{i,1}^\Delta(t)$.
By simply using the bound for $U_i$ given by \eqref{boundUi}, we can achieve
$$|I_{i,1}^\Delta(t)|\leq C_0 E_i^\Delta(t)^{\frac 12}\mathcal{E}_i^\Delta(t)^{\frac 12}.$$
The other terms in $I_i^\Delta(t)$ can be estimated exactly along the proof of
Proposition \ref{propEdifference}. Then for small $t$, we deduce 
\begin{equation}\begin{split}\label{ineqIhomo}
|I_i^\Delta(t)|&\leq C_0 E^\Delta(t)^{\frac 12}\mathcal{E}^\Delta(t)^{\frac 12}
+C_0\bigl(E_i^{(1)}(t)^{\frac12}+E^{(2)}(t)^{\frac 14}\bigr)\mathcal{E}^\Delta(t)\\
&\leq \frac{\delta}{2n}\mathcal{E}^\Delta(t)+ C_0 E^\Delta(t)
+C_0\bigl(E_i^{(1)}(t)^{\frac12}+E^{(2)}(t)^{\frac 14}\bigr)\mathcal{E}^\Delta(t).
\end{split}\end{equation}
Substituting \eqref{ineqAhomo} and \eqref{ineqIhomo} into \eqref{equalitydtEi},
and summing up over $i$, leads to
\begin{equation}\label{ineqdtEhomo}
t\frac{d}{dt} E^\Delta(t)\leq -\frac{\delta}{2}\mathcal{E}^\Delta(t)+ C_0 E^\Delta(t)
+C_0\bigl(E^{(1)}(t)^{\frac12}+E^{(2)}(t)^{\frac12}+E^{(2)}(t)^{\frac 14}\bigr)\mathcal{E}^\Delta(t).
\end{equation}
The bound \eqref{rateEi} guarantees the existence of a positive time $t_2$, such that
for all $0<t<t_2$, there holds
$C_0\bigl(\sqrt t+E^{(1)}(t)^{\frac12}+E^{(2)}(t)^{\frac12}+E^{(2)}(t)^{\frac 14}\bigr)\leq \frac{\delta}2.$
Then \eqref{ineqdtEhomo} turns into
$$t\frac{d}{dt} E^\Delta(t)\leq C_0 E^\Delta(t),$$
hence
\begin{equation}\label{ineqEhomo}
E^\Delta(t)\leq\left(\frac{t}{t'}\right)^{C_0}E^\Delta(t'),\quad\forall 0<t'<t.
\end{equation}
In view of \eqref{estimateEDelta}, the right-hand side of \eqref{ineqEhomo} converges to $0$ as
$t'\rightarrow0$. Thus $E^\Delta(t)=0$, which means that $f^{(1)}(t)=f^{(2)}(t)$ for all $0<t<\min(t_1,t_2)$.
Returning to the original variables, we conclude that $\oto(t)=\ott(t)$ for all $0<t<\min(t_1,t_2)$.
Then the desired uniqueness follows from the global well-posedness result established in Theorem $1.1$
of \cite{GS15}, and the whole theorem has been proved.
\end{proof}

\smallskip
\noindent {\bf Acknowledgements.}
The first author (G.L.) warmly thanks Ping Zhang of the Chinese Academy of Sciences for his generous invitation to collaborate and fruitful discussions which led to the present paper.
The hospitality of the Morningside center of mathematics is also gratefully acknowledged.

\end{document}